\documentclass[11pt,reqno]{amsart}
\usepackage{amsmath}
\usepackage{amsthm}
\usepackage{amssymb}
\usepackage{enumerate}
\usepackage{amscd}
\usepackage{color}
\usepackage{pb-diagram}
\usepackage{graphicx}
\usepackage[all, cmtip]{xy}
\usepackage{soul}
\usepackage{esint}
\usepackage{hyperref}
\hypersetup{pdftex,colorlinks=true,allcolors=blue}
\usepackage{hypcap}
\usepackage[titletoc]{appendix}

\theoremstyle{plain}

\newtheorem{lemma}{Lemma}

\newtheorem{coro}[lemma]{Corollary}
\newtheorem{theorem}[lemma]{Theorem}

\newtheorem*{prop*}{Proposition}
\theoremstyle{definition}

\newtheorem{remark}[lemma]{Remark}

\theoremstyle{remark}

\begin{document}
\title[Holomorphic curvature and canonical bundle]{A remark on our paper ``Negative Holomorphic curvature and positive canonical bundle''}

\author{Damin Wu}
\address{Department of Mathematics\\
University of Connecticut\\
196 Auditorium Road,
Storrs, CT 06269-3009, USA}

\email{damin.wu@uconn.edu}

\author{Shing--Tung Yau}
\address{Department of Mathematics \\
				Harvard University \\
				One Oxford Street, Cambridge MA 02138}
\email{yau@math.harvard.edu}
\thanks{The first author was partially supported by the NSF grants DMS-1308837 and DMS-1611745. The second author was partially supported by the NSF grant DMS-0804454.}
\maketitle


\maketitle

This is a continuation of our first paper in \cite{Wu-Yau:2015}.
There are two purposes of this paper: One is to give a proof of the main result in \cite{Wu-Yau:2015} without going through the argument depending on numerical effectiveness.
The other one is to provide a proof of our conjecture, mentioned in \cite{Tosatti-Yang:2015}, where the assumption of negative holomorphic sectional curvature is dropped to quasi-negative.
We should note that a solution to our conjecture is also provided by Diverio-Trapani~\cite{Diverio-Trapani:2016}. Both proofs depend on our argument in \cite{Wu-Yau:2015}.
But our argument here makes use of the argument given by the second author and Cheng in \cite{Cheng-Yau:1975}.

The proof of Theorem~\ref{th:WY1} below is obtained by us in 2015, and has been distributed in the community; e.g. it was presented by the second author in the birthday conference of Richard Schoen on June 21, 2015. We have also settled the complete noncompact case, which, together with applications, will appear soon.

\section{Negative holomorphic sectional curvature}
By using the argument in \cite{Wu-Yau:2015}, we can directly derive the following result in the K\"ahler setting, without using the notion of nefness. 
\begin{theorem} \label{th:WY1}
Let $(X, \omega)$ be a compact K\"ahler manifold with negative holomorphic sectional curvature. Then $X$ admits a K\"ahler-Einstein metric of negative Ricci curvature. In particular, the canonical bundle of $X$ is ample.
\end{theorem} 
\begin{proof}
  For $t > 0$, consider the Monge-Amp\`ere type equation
  \begin{equation*} 
  (t \omega + dd^c \log \omega^n + dd^c u_t)^n = e^{u_t} \omega^n. \tag*{$(\textup{MA})_t$}
  \end{equation*}
  Since $\omega>0$ and $X$ is compact, there exists a sufficiently large constant $t_1 > 1$ such that $t_1 \omega + dd^c \log \omega^n > 0$ on $X$.
  Fix an nonnegative integer $k$ and $0 < \alpha <1$. Denote by $C^{k,\alpha}(X)$ the H\"older space with respect to fixed metric $\omega$. Define
  \begin{equation} \label{eq:defI}
    \begin{split}
     I  = \{t \in [0, t_1] \mid \; &  \textup{there is a $u_t \in C^{k+2,\alpha}(X)$ satisfying $(\textup{MA})_t$ and} \\
              & \textup{$t \omega + dd^c \log \omega^n + dd^c u_t > 0$}\}.
   \end{split}
  \end{equation}
  First, note that $I \ne \emptyset$, since $t_1 \in I$. Indeed, $(\textup{MA})_{t_1}$ can be written as
  \[
      (t_1 \omega + dd^c \log \omega^n + dd^c u_{t_1})^n = e^{u_{t_1} + f_1} (t_1 \omega + dd^c \log \omega^n)^n, 
  \]
  where 
  \[
      f_1 = \log \frac{\omega^n}{(t_1 \omega + dd^c \log \omega^n)^n} \in C^{\infty}(X).
  \]
  Applying \cite{Yau:1978:Calabi} obtains a solution $u_{t_1} \in C^{\infty}(X)$.
  
  That $I$ is open in $[0, t_1]$ follows from the implicit function theorem. Indeed, let $t_0 \in I$ with corresponding function $u_{t_0} \in C^{k+2,\alpha}(X)$. Then, there exists a small neighborhood $J$ of $t_0$ in $[0, t_1]$ and a small neighborhood $U$ of $u_{t_0}$ in $C^{k+2,\alpha}(X)$ such that
 \[
    t \omega + dd^c \log \omega^n + dd^c v > 0 \quad \textup{for all $t \in J$ and $v \in U$}.
 \]
 Define a map $\Phi: J \times U \to C^{k,\alpha}(X)$ by
  \[
     \Phi (t, v) = \log \frac{(t \omega + dd^c \log \omega^n + dd^c v)^n}{\omega^n} - v.
  \]
Then, the linearization is given by
\[
   \Phi_{u_{t_0}} (t_0, u_{t_0}) h = \left. \frac{d}{ds} \Phi(t_0, u_{t_0} + s h) \right|_{s = 0} = (\Delta_{t_0} - 1) h
\]
which is invertible from $C^{k+2,\alpha}(X)$ to $C^{k,\alpha}(X)$, where $\Delta_{t_0}$ stands for the Laplacian of metric $t_0 \omega + dd^c \log \omega^n + dd^c u_{t_0}$. Thus, we can apply the implicit function to obtain the openness of $I$.

The closedness is contained in the proof of Theorem 7 in \cite{Wu-Yau:2015}. More precisely, for $t \in I$ we denote
\[
   \omega_t = t \omega + dd^c \log \omega^n +dd^c u_t, \quad \textup{and} \quad S = \frac{n \omega_t^{n-1} \wedge \omega}{\omega_t^n},
\] 
Then $(\textup{MA})_t$ becomes $\omega_t^n = e^{u_t} \omega^n$. It follows that
\[
   dd^c \log \omega_t^n = dd^c \log \omega^n + dd^c u_t = \omega_t - t \omega.
\]
It follows from \cite[Proposition 9]{Wu-Yau:2015} that
\begin{equation} \label{eq:WY-Sch}
   \Delta' \log S \ge \Big[\frac{t}{n} + \frac{(n+1)\kappa}{2n} \Big] S - 1,
\end{equation}
where $\Delta'$ is the Laplacian of $\omega_t$ and $\kappa>0$ is a constant such that $- \kappa$ is the upper bound of the holomorphic sectional curvature $H(\omega)$ of $\omega$. Thus,
\[
   S \le \frac{2n}{\kappa (n+1)}.
\]
On the other hand, applying the maximum principle to $(\textup{MA})_t$ yields 
\[
  \max_X u_t \le C,
\] 
where $C>0$ is a generic constant independent of $t$.
It follows from the same process in \cite{Wu-Yau:2015} that 
\[
   \| u_t \|_{C^{k+2,\alpha}} \le C.
\]
This shows the closedness of $I$. In particular, $0 \in I$ with corresponding $u_0 \in C^{\infty}(X)$.  This gives us the desired K\"ahler-Einstein metric $dd^c \log \omega^n + dd^c u_0$.
\end{proof}

\section{Quasi-negative holomorphic sectional curvature}
In this section we seek to extend Theorem~\ref{th:WY1} to the case $(X, \omega)$ has \emph{quasi-negative} holomorphic sectional curvature, i.e., the holomorphic sectional curvature of $\omega$ is less than or equal to zero everywhere on $X$ and is strictly negative at one point of $X$. We can prove the following result (compare \cite{Diverio-Trapani:2016}):
\begin{theorem} \label{th:WY2}
Let $(X, \omega)$ be a compact K\"ahler manifold with quasi-negative holomorphic sectional curvature. Then
\begin{equation} \label{eq:intc1n}
    \int_X c_1 (K_X)^n > 0.
\end{equation}
\end{theorem}
The proof uses the following two lemmas.
\begin{lemma} \label{le:HoTi}
  Let $\Theta$ be a $(1,1)$ form on a compact K\"ahler manifold $(X, \omega)$ which admits a smooth potential on every coordinate chart $U$, i.e., $\Theta = dd^c h_U$ for some $h_U \in C^{\infty}(U)$. Then for any smooth function $u$ on $X$ satisfying $\Theta + dd^c u \ge 0$ on $X$,
  \[
     \int_X e^{-\beta (u - \max_X u)} \omega^n < C,
  \]
  where $\beta$ and $C$ are positive constants depending only on $\Theta$ and $\omega$.
\end{lemma}

The proof of Lemma~\ref{le:HoTi} is essentially a globalization of the local estimate in \cite[p. 97, Theorem 4.4.5]{Hormander:1990} via the Green's formula on a compact manifold, which becomes standard by now. In fact, Lemma~\ref{le:HoTi} has been known to the second author since the late 1970s.

The next lemma can be viewed as a special case of the second author with Cheng \cite[p. 335, Theorem 1]{Cheng-Yau:1975}.
\begin{lemma} \label{le:CYlog-v}
Let $v$ be a negative $C^2$ function on a compact K\"ahler manifold $(X, \omega)$. Suppose $\Delta_{\omega} v \ge - \varphi$ for some continuous function $\varphi \ge 0$ on $X$. Then,
\[
     \int_X |\nabla \log (-v) |^2 \omega^n \le \frac{1}{\min_X (-v)} \int_X \varphi \, \omega^n.
\]
\end{lemma}

\noindent
\emph{Proof of Lemma~\ref{le:CYlog-v}}.
 Note that
\[
   \Delta \log (-v) = \frac{\Delta v}{v} - |\nabla \log (-v)|^2.
\]
Integrating both sides against $\omega^n$ over $X$ yields
\[
     \int_X |\nabla \log (-v)|^2 = \int_X \frac{-\Delta v}{-v} \le \int_X \frac{\varphi}{\min (-v)},
\]
which is the desired estimate.
\qed

\begin{coro} \label{co:WYcpt}
Let $\Theta$ and $u$ satisfy the condition of Lemma~\ref{le:HoTi}. That is, let $\Theta$ be a $(1,1)$ form on a compact K\"ahler manifold $(X, \omega)$ admitting a smooth potential on every coordinate chart $U$, and let $u \in C^{\infty}(X)$ satisfy $\Theta + dd^c u \ge 0$ on $X$. Then $v \equiv u - \max_X u - 1$ satisfies
\begin{equation} \label{eq:W12log-v}
    \int_X |\log (-v)|^2 \omega^n + \int_X |\nabla \log (-v) |^2 \omega^n \le C
\end{equation}
where $C>0$ is a constant depending only on $\Theta$ and $\omega$. Consequently, for any sequence $\{u_l\}_{l=1}^{\infty}$ of smooth functions on $X$ satisfying $\Theta + dd^c u_l \ge 0$, the sequence 
\[
   \big\{ \log (1 + \max_X u_l - u_l)\big\}_{l=1}^{\infty}
\]
is relatively compact in $L^2(X)$. 
\end{coro}

\noindent
\emph{Proof of Corollary~\ref{co:WYcpt}}.
  By the hypothesis $\Theta + dd^c u \ge 0$ we have
\[
    \Delta_{\omega} u \ge - \textup{tr}_{\omega} \Theta.
\]
Applying Lemma~\ref{le:CYlog-v} to $v = u - \max_X u - 1$ yields
\[
    \int_X |\nabla \log (-v)|^2 \omega^n \le \int_X \textup{tr}_{\omega} \Theta \, \omega^n,
\]
since $\min_X (-v) = 1$. To bound the $L^2$-norm of $\log (-v)$, recall Lemma~\ref{le:HoTi} that
\[
    \int_X e^{-\beta (u - \max_X u - 1)} \omega^n \le C \int_X e^{\beta} \omega^n
\]
where the constants $C>0$ and $\beta>0$ depend only on $\Theta$ and $\omega$. Observe that
\[
    t \ge \log t, \quad \textup{for all $t \ge 1$}.
\] 
 Then,
\[
    e^t \ge \frac{t^N}{N!} \ge \frac{(\log t)^N}{N!} \quad \textup{for all $N \ge 0$, $t \ge 1$}.
\]
Hence,
\[
    e^{\beta (-v)} \ge \frac{[\log (-v)]^{N \beta }}{(N!)^{\beta}}.
\]
Choose a large integer $N$ such that $N \beta \ge 2$. Since $\beta$ depends only on $\Theta$ and $\omega$, so is $N$. Applying H\"older's inequality, if needed, yields
\[
   \int_X [\log (-v)]^2 \omega^n \le C_1,
\]
where $C_1>0$ is a constant depending only on $\Theta$ and $\omega$. Therefore, we have established inequality \eqref{eq:W12log-v}, which implies the second statement, in view of the standard Rellich Lemma (see \cite[p. 37, Theorem 2.9]{Hebey:1999} for example).
\qed

Let us now proceed to show Theorem~\ref{th:WY2}.

\noindent
\emph{Proof of Theorem~\ref{th:WY2}.}
We shall use the continuity method, as in the proof of Theorem~\ref{th:WY1}. Let $u_t \in C^{\infty}(X)$, $t> 0$, be the solution of 
\[
   \omega_t^n = (t \omega + dd^c \log \omega^n + dd^c u_t)^n = e^{u_t} \omega^n
\]
with $\omega_t \equiv t \omega + dd^c \log \omega^n + dd^c u_t > 0$ on $X$. Let $I$ be the interval defined by \eqref{eq:defI} in the previous proof. Then, $I$ contains a large number $t_1$ and $I$ is open, by the same argument in the proof of Theorem~\ref{th:WY1}.

We \textbf{claim} that $I$ contains every $t$ in the interval $(0, t_1]$ provided that the holomorphic sectional curvature $H(\omega) \le 0$ (which, in particular, implies that $K_X$ is nef if $H (\omega) \le 0$; compare \cite{Tosatti-Yang:2015}). Fix an arbitrary $t_2 \in (0, t_1)$. For any $t_2 \le t \le t_1$, by \eqref{eq:WY-Sch} we have
\begin{align*}
   \Delta' \log S 
    \ge \Big[ \frac{t}{n} + \frac{(n+1)\kappa}{2n} \Big] S - 1 \ge \frac{t}{n} S - 1,
\end{align*}
since $H(\omega) \le - \kappa$ with $\kappa \ge 0$ on $X$. By the standard maximum principle,
\[
   \max_X S \le \frac{n}{t} \le \frac{n}{t_2}, \quad \textup{for all $t \ge t_2$}.
\]
On the other hand, applying the maximum principle to the Monge-Amp\`ere equation yields
\begin{equation} \label{eq:maxu}
   \max_X u_t \le C \quad \textup{for all $t > 0$}.
\end{equation}
Here and below, we denote by $C>0$ a constant depending only on $n$ and $\omega$, unless otherwise indicated.
Processing as \cite{Wu-Yau:2015} yields
\[
    C t_2 \, \omega \le \omega_t \le \frac{C}{t_2^{n-1}} \omega.
\]
and
\[
    \inf_X u_t \ge C \log t_2.
\]
This implies the estimate for $u_t$ up to the second order. The estimate constants may depend on $t_2$, which is fixed. One can then apply, either the local H\"older estimate of the second order, or the third order estimate, to obtain a bound for the $C^{2,\alpha}(X)$-norm of $u_t$; and hence a bound for $C^{k,\alpha}(X)$-norm of $u_t$. The bound may depend on $t_2$. This show the closedness of $I$ up to the subinterval $[t_2, t_1]$. In particular, $t_2 \in I$. Since $t_2$ is arbitrary, we have proven the claim.

Using the claim, to prove \eqref{eq:intc1n} it is equivalent to show
\begin{equation} \label{eq:limsupVt}
   \limsup_{t \to 0} \int_X \omega_t^n > 0, 
\end{equation}
since
\[
   \int_X \omega_t^n = \int_X c_1(K_X)^n +  n t \int_X c_1(K_X)^{n-1} \wedge \omega + O(t^2)
\]
as $t \to 0$. Using inequality \eqref{eq:WY-Sch} again we have
\[
   \Delta' \log S \ge \frac{(n+1)\kappa}{2n} S - 1 \ge \frac{(n+1)\kappa}{2} 
  \exp \left(-\frac{\max_X u_t}{n}\right) - 1, 
\]
where we apply the Netwon-Maclaurin's inequality 
\[
    S = \frac{\sigma_{n-1}}{\sigma_n} \ge n \sigma_n^{-1/n} = n e^{-u_t / n}, \quad t > 0.
\]
Integrating against $\omega_t^n$ over $X$ yields
\begin{align}
    \exp\Big(- \frac{\max_X u_t}{n}\Big)
    & \le \frac{\int_X \omega_t^n}{\frac{n+1}{2} \int_X \kappa \, \omega_t^n} \notag \\
    & = \frac{\int_X \exp(u_t - \max_X u_t -1) \omega^n}{\frac{n+1}{2}\int_X  \kappa  \exp(u_t - \max_X u_t - 1) \, \omega^n}. \label{eq:intmaxu}
\end{align}
Since $t_1\omega + dd^c \log \omega^n + dd^c u_t \ge \omega_t > 0$ for each $0< t \le t_1$, applying Corollary~\ref{co:WYcpt} with $\Theta = t_1 \omega + dd^c \log \omega^n$ yields that the set
\[
   \Big\{ \log \big(1 + \max_X u_t - u_t \big); 0 < t \le t_1\Big\}
\]
is relatively compact in $L^2(X)$. Then, a sequence $\{ \log (1 + \max_X u_{t_l} - u_{t_l})\}_{l=1}^{\infty}$ converges in $L^2(X)$ to a function $w \in L^2(X)$. The standard $L^p$ theory (cf. \cite[p. 30, Corollary 2.7]{Adams-Fournier:2003} for example) implies that a subsequence, still denoted by $\{ \log (1 + \max_X u_{t_l} - u_{t_l})\}$, converges to $w$ almost everywhere on $X$. It follows that 
\[
    \exp(u_{t_l} - \max_X u_{t_l} - 1) \longrightarrow \exp(-e^w)
\]
almost everywhere on $X$, as $l \to +\infty$. Since $\exp( u_{t_l} - \max_X u_{t_l} - 1) \le 1$ for each $l$, it follows from the Lebesgue dominated convergence theorem that $\exp (-e^w) \in L^1(X)$ and
\[
   \int_X \exp ( u_{t_l} - \max_X u_{t_l} - 1) \, \omega^n 
    \longrightarrow \int_X \exp(-e^w) \, \omega^n > 0.
\]
Moreover, since $\kappa \ge 0$ on $X$ and $\kappa > 0$ in a neighborhood of $X$, we have
\[
   \int_X \kappa \exp (u_{t_l} - \max_X u_{t_l} - 1) \, \omega^n
    \longrightarrow \int_X \kappa \exp(-e^w) \, \omega^n > 0.
\]
Plugging these back to \eqref{eq:intmaxu} yields
\begin{equation} \label{eq:infmaxu}
   \lim_{l \to +\infty} \big(\max_X u_{t_l}\big) \ge - C,
\end{equation}
where the constant $C>0$ depends only on $n$ and $\omega$. 

From \eqref{eq:infmaxu} we have \emph{two slightly different ways} to conclude \eqref{eq:limsupVt}. For the first way, note that combining \eqref{eq:infmaxu} and the previous upper bound \eqref{eq:maxu} we know the constant sequence $\{ \max_X u_{t_l} \}_{l=1}^{\infty}$ is bounded. Thus, by further passing to a subsequence we can assume $u_{t_l}$ converges to $(-e^w+c)$ almost everywhere on $X$, where $c$ is a constant. By the upper bound \eqref{eq:maxu} we can apply the Lebesgue dominated convergence theorem to conclude
\[
   \int_X \omega_{t_l}^n = \int_X e^{u_{t_l}} \omega^n \longrightarrow \int_X \exp(- e^w + c) \, \omega^n > 0,
\] 
as $l \to +\infty$. This proves the desired \eqref{eq:limsupVt}.

A second way to see \eqref{eq:limsupVt} is to plug \eqref{eq:infmaxu} back to the integral inequality in Lemma~\ref{le:HoTi} (with $\Theta = t_1 \omega + dd^c \log \omega^n$). This gives
\[
   \int_X e^{-\beta u_{t_l} } \omega^n \le C \quad \textup{for all $l \ge 1$},
\]
where $\beta>0$ is a constant depending only on $\omega$. It follows that
\begin{align*}
    \int_X \omega^n 
    & = \int_X e^{-\frac{\beta}{\beta +1} u_{t_l}} \, e^{\frac{\beta}{\beta+1} u_{t_l}} \omega^n \\
    & \le \Big(\int_X e^{-\beta u_{t_l}} \Big)^{\frac{1}{\beta+1}} \Big(\int_X e^{u_{t_l}} \omega^n\Big)^{\frac{\beta}{\beta+1}}  \quad \textup{(by H\"older's inequality)} \\
    & = C^{\frac{1}{\beta+1}} \Big(\int_X \omega_{t_l}^n \Big)^{\frac{\beta}{\beta+1}}.
\end{align*}
Hence,
\[
   \int_X \omega_{t_l}^n \ge C^{-1/\beta} \Big(\int_X \omega^n \Big)^{\frac{\beta+1}{\beta}} > 0, \quad \textup{for all $l \ge 1$}.
\]
This again confirms \eqref{eq:limsupVt}. The proof of Theorem~\ref{th:WY2} is therefore completed.
\qed

\begin{remark}
 If $X$ is projective then by \cite[Lemma 5]{Wu-Yau:2015} we know $K_X$ is ample. When $X$ is K\"ahler, Diverio-Trapani~\cite{Diverio-Trapani:2016} points out that \eqref{eq:intc1n} implies $X$ is Moishezon, by using a result of \cite{Demailly-Paun:2004}. Then a theorem of Moishezon tells us that $X$ is projective (see \cite[p. 442, Theorem 2]{Hartshorne:1977} for example); hence, $K_X$ is ample by reducing to the projective case.
\end{remark}

\bibliographystyle{alpha}
\bibliography{../../../Bib/DWu}

\end{document}